\documentclass[a4paper, 12 pt, ]{article}  % Comment this line out if you need a4paper

\usepackage{graphics} % for pdf, bitmapped graphics files
\usepackage{epsfig} % for postscript graphics files
\usepackage{graphicx}
\usepackage{mathtools}
\usepackage{times}
\usepackage{epstopdf}
\usepackage{csquotes}
\usepackage{color}
\usepackage{amssymb}
\usepackage{amsthm}
\usepackage{float}
\usepackage{pgfplots}
\pgfplotsset{compat=1.5}
\usepackage{subfig}
\usepackage[colorlinks,allcolors=blue]{hyperref}
\usepackage{cite}
\newtheorem{theorem}{Theorem}[section]

\newtheorem{proposition}[theorem]{Proposition}
\newtheorem{remark}{Remark}
\usepackage{color}

\usepackage{kbordermatrix}
\def\be{\begin{equation}}
\def\ee{\end{equation}}
\def\bea{\begin{eqnarray}}
\def\eea{\end{eqnarray}}

\def\l{\label}

\newcommand{\ncom}{\newcommand}
\ncom{\beqn}{\begin{eqnarray*}} \ncom{\eeqn}{\end{eqnarray*}}
\ncom{\beq}{\begin{eqnarray}} \ncom{\eeq}{\end{eqnarray}}
\newcommand{\kc}[1]{\textcolor{black}{#1}}

\newcommand{\kcb}[1]{\textcolor{black}{#1}}
\makeatletter
\title{\Large \bf  Stability Analysis of Constrained Optimization Dynamics via Passivity Techniques}
\author{ K. C. Kosaraju$^1$, V. Chinde$^{2}$, R. Pasumarthy$^1$, A. Kelkar$^{2}$ and  N. M. Singh$^{3}$%
	\thanks{ $^1$Department
		of Electrical Engineering at IIT-Madras, Chennai, India {\tt\small ee13d015, ramkrishna@ee.iitm.ac.in}.}
	\thanks{$^2$Department of Mechanical Engineering at Iowa State University, Ames, USA {\tt\small vchinde@iastate.edu, akelkar@iastate.edu}.}
	\thanks{$^3$Department of Electrical Engineering at VJTI, Mumbai, India {\tt \small nmsingh59@gmail.com}.
	}%
}

\begin{document}

	\maketitle
	\thispagestyle{empty}
	\pagestyle{empty}
	
\begin{abstract}
	%In this paper we identify new passive maps in	primal-dual gradient dynamics
	In this paper, we present passivity based convergence analysis of continuous time primal-dual gradient method for convex optimization problems. We first show that a convex optimization problem with only affine equality constraints admit a Brayton Moser formulation. This observation leads to a new passivity property derived from a Krasovskii type storage function. Secondly, the inequality constraints are modeled as a state dependent switching system. Using hybrid methods, it is shown that each switching mode is passive and the passivity of the system is preserved under arbitrary switching. Finally, the two systems, (i) one derived from the Brayton Moser formulation and (ii) the state dependent switching system, are interconnected in a power conserving way. The resulting trajectories of the overall system are shown to converge asymptotically, to the optimal solution of the convex  optimization problem. The proposed methodology is applied to an energy management problem in buildings and  simulations are provided for corroboration. 
% 	%We next extend it to inequality constraint by proving the passivity for the 
% 	We next show that the switching dynamics corresponding to the inequality constraint are passive. 
% 	In this paper we present a new passivity analysis for primal-dual gradient dynamics of convex optimization problem. 
% 	with the help of Brayton Moser framework, we developed storage functions of Krasovskii type that resulted in the new passivity property.
% 	In this paper starting with Brayton Moser frame work we present a new passivity for primal-dual gradient dynamics 
% 	 This result is extende
	 
	%	Increasing energy demand, with supply constraints brings consumers and producers to behave in a manner of maximizing social welfare. In the context of building energy management system, the buildings (consumers) try to optimize their set-point to maximize comfort, whereas utilities try to reduce their generation costs. In this paper, the resulting optimization problem is expressed as an interconnection of a continuous time dynamical system, which naturally admits a Brayton-Moser form, together with a given constraint which can be represented as a  switched passive system. The solution to the optimization problem is established via the stability  of the interconnected system. The results are applied to energy management problem in buildings and supported by simulations.
	\end{abstract}
	\section{Introduction}
	    The applications of convex optimization are ubiquitous in various fields of research \cite{ben2001lectures} such as, resource allocation \cite{ibaraki1988resource}, utility maximization \cite{kelly1998rate} etc. Numerous methods are proposed to solve these optimization problems \cite{boyd2004convex}. Solution techniques in a distributed setting have gained importance in recent times\cite{xiao2006optimal}.  One  of  the standard  tools  for  designing  algorithms  to  solve  such  optimization  problems  is  through  primal-dual  gradient  method \cite{kose1956solutions,arrow1958studies}. Gradient based methods are a well known class of mathematical routines for solving convex optimization problems. From a control and dynamics perspective, these primal dual algorithms have much to offer in terms of using tools from system theory to have better understanding of the underlying dynamics. Passivity, a system theoretic tool has been widely used for studying dynamical systems as it relates to  energy conservation and used for attaining classical control objectives such as stabilization and performance. %A control theoretic view point of centralized/distributed convex optimization problem using primal dual dynamics is described in \cite{wang2011control}.
	    
	    The convergence of gradient based methods and Lyapunov stability relate the solution of the optimization problem to the equilibrium point of a dynamical system. The Krasovskii-Lyapunov function is  particularly suited for establishing stability of the continuous time gradient laws, as the equilibrium point (or solution of the optimization problem) is not known apriori.	In \cite{feijer2010stability}, the authors used this Krasovskii Lyapunov function and hybrid Lasalle's invariance principle \cite{lygeros2003dynamical} to prove asymptotic stability of a network optimization problem. The gradient structure of the primal-dual equations characterizing the optima of a  convex optimization with only equality constraint admit a Brayton Moser (BM) form. Further, using the duality between energy and co-energy the BM form is partially transformed  into a port-Hamiltonian (pH) form \cite{jeltsema2009multidomain,stegink2016unifying}. These transformations pave the way for passivity/stability analysis using (i) the invariance principle for discontinuous Caratheodory systems \cite{cherukuri2016asymptotic} and (ii) an incremental passivity property for the misfit dynamics. \textcolor{black}{In \cite{simpson2016input}, the authors provided robustness analysis for  primal-dual dynamics of convex optimization problem with only equality constraint. A brief list of applications of primal dual gradient methods is given below: in the context of power systems, these methods have been used to achieve optimal load sharing \cite{yi2015distributed}, investigating effects of real-time pricing on stability and volatility of electricity markets \cite{roozbehani2012volatility}, and stability analysis of integrated power markets with physical dynamics\cite{stegink2016unifying}. In \cite{ma2016energy}, primal-dual gradient method is used to solve the energy management problem in application to HVAC systems.}
	    
		Increasing energy demand, with supply constraints brings consumers and producers to behave in a manner of maximizing social welfare.	%The application of primal dual methods to the social welfare problem is of interest to the authors as this problem has its roots in power systems, building systems, economics, etc.
	 In the context of building energy management system, the buildings (consumers) try to optimize their set-point to maximize comfort, whereas utilities try to reduce their generation costs.
	%	\textcolor{black}{Supply-demand balance is essential for stable operation of a power system. Increasing demand and rapid penetration of distributed energy resources into the existing grid induces more uncertainties towards safe operation of the grid.} 
	Traditionally, conventional generators were employed to meet the additional demand. With increasing demand side management programs, utilities provide incentives to consumers to lower the overall power demand. This results in reducing load during the time when prices are high. These programs lead to a process which involve both the supply and the demand-side resources to minimize the overall cost.
		%\textcolor{black}{Demand side management programs include usage of energy efficient appliances and smart adaption of variable pricing strategies.} 
		\textcolor{black}{Building systems being  one of the strong contenders for providing ancillary services to the grid \cite{hao2013ancillary} %\textcolor{black}{The advent of smart grid technologies provide sophistication to consumers and providers to schedule supply and demand at regular intervals of time.}
		in which  heating ventilating and airconditioning (HVAC) systems play a significant role, as they account for approximately 30-40\% of the total energy demand.} %\cite{perez2008review}.\textcolor{black}{ The idea of  minimizing the energy costs  of producer and user discomfort is formulated as a social welfare problem \cite{chen2012optimal}.}
		There is a vast amount of literature on demand response (DR) strategies in maintaining optimal balance between supply and demand at all times. An overview on the types of DR	and taxonomy for demand side management is described in \cite{palensky2011demand}. Real-time pricing based demand response (DR) application \cite{mathieu2013state} has been deployed in smart meters of residential homes to have direct control of loads such as air conditioners etc. Recently the authors \cite{chinde2016building} have proposed the Brayton - Moser (BM) formulation for modeling and stability analysis of HVAC subsystems.	%\textcolor{black}{The use of Krasovskii  method  of  Lyapunov  function is  particularly suited  for  saddle  point  dynamics (primal-dual equations),  in  which  the  symmetry of  gradient  dynamics  simplifies  the  analysis  and  is a general  method  for  gradient  based  systems. In \cite{feijer2010stability}, the authors used this Krasovskii Lyapunov function and hybrid Lasalle's invariance principle \cite{lygeros2003dynamical} to prove asymptotic stability of a network optimization problem. The well known BM \cite{jeltsema2009multidomain} framework, for physical systems modeling, admits a generalized gradient structure \cite{van2011relation}. The gradient structure of the primal-dual equations naturally admit a BM form.  Using the duality between energy and co-energy \cite{jeltsema2009multidomain}, the authors in \cite{stegink2016unifying} transformed these BM dynamics, partially into a port-Hamiltonian (pH) form. Stability analysis was presented using the invariance principle for discontinuous Caratheodory systems \cite{cherukuri2016asymptotic} and an incremental passivity property (for the misfit dynamics). }
	%Primal-dual methods have been used to solve the network optimization problems as detailed in \cite{feijer2010stability} and also stability guarantees have been provided in \cite{cherukuri2016asymptotic}.
	%The paper \cite{stegink2016unifying} is of interest to authors where a unified energy-based framework using systemic property of passivity, for stability analysis of power grids with market dynamics is proposed. The key idea was to model the power system and the market dynamics within the port-Hamiltonian framework.%, which has its. 
	%
	%Recently the authors \cite{chinde2016building} have proposed the dual of port-Hamiltonian systems known as Brayton - Moser systems for modeling and stability analysis of HVAC subsystems.
	\subsection*{Motivations and Main contributions}
		%%%
% 			\textcolor{black}{
In any stabilization problem, whether it is to stabilize a system to an equilibrium point or to an operating point, the velocities must converge asymptotically to zero. This observation motivates the need for storage functions  defined explicitly in velocities. A good candidate, in general, is a positive definite quadratic function of velocities. In \cite{ICCvdotidot}, the authors showed that for systems specified in BM form a new passivity property can be derived, with differentiation at both the port variables. In this paper, we employ a similar methodology to derive passive maps directly from the BM form of a convex optimization problem with only equality constraints. The primal-dual dynamics of the inequality constraint is modelled as a state dependent switching system. We first show that each switching mode is passive and the passivity of the system is preserved under arbitrary switching using hybrid passivity tools, a methodology similar to switched Lyapunov functions for stability analysis of switch system. Finally, the two systems, (i) one derived from the Brayton Moser formulation and (ii) the state dependent switching system, are interconnected in a manner such that the equilibrium is the solution of the convex  optimization problem.
	As a case study, we apply primal-dual methodology to a social welfare problem associated with building energy management system. From an application stand point, BM framework presents a design methodology for stabilization \cite{chinde2016building} of HVAC subsystems. This motivates us to analyze the social-welfare problem, in the context of building systems, formulated as a trade-off between user comfort and generation costs. 
	
	An elaborate version of this manuscript can be found at \cite{css_arxiv}.
	\section{Preliminaries}
	\subsection*{Brayton-Moser formulation}
	\textcolor{black}{
 Consider the standard representation of a dynamical system in Brayton-Moser (BM) formulation
\begin{equation}
Q(x)\dot{x}=\nabla_x P(x)+G(x)u\label{BM}
\end{equation}
the system state vector $x \in \mathbb{R}^{n}$ and the input vector $u \in \mathbb{R}^{m}$ ($m \leq n$). $P(x): \mathbb{R}^{n} \rightarrow \mathbb{R}$ is a scalar function of the state, which has the units of power\kc{,} also referred to as mixed potential function %Since the formulation originated from the study of nonlinear electrical networks, the mixed potential function is the combination
%of content and co-content functions and the power transfer
%between the capacitor and inductor sub systems
\cite{jeltsema2009multidomain}\kc{.} $Q(x): \mathbb{R}^{n} \rightarrow \mathbb{R}^{n\times n}$ and $G(x): \mathbb{R}^{n} \rightarrow \mathbb{R}^{n\times m}$. 
In BM formulation we represent the system dynamics in pseudo-gradient form, ($Q(x)$ and $P(x)$ are indefinite). Therefore $P(x)$ can not be used as a Lyapunov function for stability analysis. %There are multiple ways of constructing a Lyapunov function, most of them converge in finding $\alpha\in \mathbb{R} $ and $M\in \mathbb{R}^{n\times n}$ \cite{ortega2003power,ICCvdotidot} such that
%\cite{ortega2003power,ICCvdotidot} 
A way of constructing a suitable Lyapunov function involves finding $\alpha\in \mathbb{R} $ and $M\in \mathbb{R}^{n\times n}$ \cite{ortega2003power,ICCvdotidot} such that 
\begin{eqnarray}\label{common_storage_fun}
\tilde{P}=\alpha P+\frac{1}{2}\nabla_xP^\top M \nabla_xP
%\tilde{P}=\alpha P+\frac{1}{2}\nabla_xP^\top M \nabla_xP.
\end{eqnarray}}
%%%%%

	\section{Passivity based formulation of the optimization problem}
	Consider the following constrained optimization problem %with equality (affine) 
	\begin{equation}\label{SOP}
	\begin{aligned}
	& \underset{x \in \mathbb{R}^{n}}{\text{minimize}}
	& & f(x)\\
	& \text{subject to}
	& & h_{i}(x)=0 \hspace{0.4cm} i = 1,\hdots , m
	\end{aligned}
	\end{equation}
	where $f: \mathbb{R}^{n} \rightarrow \mathbb{R}$ is continuously differentiable $(C^1)$ and strictly convex and $h_{i} (\in C^1): \mathbb{R}^{n} \rightarrow \mathbb{R}$ is affine. Assume 
	\begin{itemize}
	    \item[(i)]that the objective function has a positive definite Hessian $\nabla_{x}^{2}f(x)$
	    \item[(ii)]that the problem (\ref{SOP}) has a finite optimum, and \kcb{Slater's condition is satisfied (i.e., the constraints are feasible) and strong duality holds} \cite{boyd2004convex}.
	\end{itemize}
	%(i) that the objective function has a positive definite Hessian $\nabla_{x}^{2}f(x)$ and (ii) that the problem (\ref{SOP}) has a finite optimum, and \kcb{Slater's condition is satisfied (i.e., the constraints are feasible) and strong duality holds} \cite{boyd2004convex}. 
	The solution $x^{*}$ is an optimal solution to \eqref{SOP}
	if there exists $\lambda^{*} \in \mathbb{R}^{m}$ such that the following  Karush-Kuhn-Tucker (KKT) conditions are satisfied.
	\begin{eqnarray}\label{mainKKT}
	\begin{aligned}
	\nabla_{x} f(x^{*})+\sum_{i=1}^{m}\lambda_{i}\nabla_{x} h_{i}(x^{*})=0\\
	h_{i}(x^{*})=0 \;\;\;\forall i \in \{1,\hdots, m\}
	\end{aligned}
	\end{eqnarray}
	The Lagrangian of \eqref{SOP} is given by
	\begin{equation}\label{MainLag}
	\mathcal{L}=f(x)+\sum_{i=1}^{m}\lambda_{i} h_{i}(x)
	\end{equation}
	Since strong duality holds for \eqref{SOP},
	$(x^{*},\lambda^{*})$ is a saddle point of the Lagrangian $\mathcal{L}$ if and only if $x^{*}$ is an optimal solution to \eqref{SOP} and $\lambda^{*}$ is optimal solution
	to its dual problem. Consider the following dynamics
	\begin{equation}\label{maindyn}
	\begin{split}
	-\tau_{x}\dot{x}&=\nabla_{x} f(x)+\sum_{i=1}^{m}\lambda_{i}\nabla_{x} h_{i}(x)+u\\
	%%%\vspace{-100mm}
	\tau_{\lambda_{i}}\dot{\lambda}_{i}&= h_{i}(x),\;\;
	y =-x. 
	\end{split}
	\end{equation}
	where $\tau_{x}, \tau_{\lambda}$ are positive definite matrices and input $u, y \in \mathbb{R}^n$. The unforced system ($u=0$) of equations \eqref{maindyn} represent primal-dual dynamics corresponding to \eqref{MainLag} and the equilibrium corresponds to the KKT conditions \eqref{mainKKT}. 
	
	%\subsection{Brayton Moser formulation and passivity analysis}
	%In convex optimization, primal-dual methods rely on finding the saddle point of the Lagrangian.
	\subsection{ The Brayton Moser formulation:} %The continuous time gradient laws \eqref{maindyn} associated with \eqref{SOP} are naturally in the BM form.
	% For simplicity we consider  the case when $m=1$.
	%\subsubsection{Brayton Moser formulation}
Denote $z= \left[x;\lambda\right]$. The continuous time gradient laws \eqref{maindyn}, associated with \eqref{SOP}, naturally admit a Brayton-Moser (BM) formulation
	\begin{equation}
	Q(z)\dot{z}=\nabla_z P(z)+u\label{BM}
	\end{equation}
	with $Q(z)=\text{diag}\{-\tau_{x},\tau_{\lambda}\}$ and $P(z)=f(x)+\lambda^{\top} h(x)$ is a scalar function of the state, which has the units of power, also referred to as mixed potential function \cite{jeltsema2009multidomain}.
	\begin{proposition}\label{prop::eq_const}
		Let $\bar{z}=(\bar{x}, \bar{\lambda})$ satisfy \eqref{mainKKT}. Assume $h(x)$ is  convex and  $f(x)$ strictly convex. Then the system of equations \eqref{maindyn} are passive with port variables $(\dot{u},\dot{y})$ \cite{ICCvdotidot}.  Further every solution of the unforced version ($u=0$) of  \eqref{maindyn} asymptotically converges to $\bar{z}$.
	\end{proposition}
	The proof of this and other propositions are given in Appendix at the end of this paper.\\ %\ref{proof::prop::eq_const}
	%\begin{proof}
	%	Considering $\tilde{P}$ in equation \eqref{common_storage_fun} with $\alpha =0$ and $M =\frac{1}{2} diag\{\tau_{x}^{-1},\tau_{\lambda}^{-1}\}$ we have 
	%	\begin{eqnarray}\label{eq_const_P}
	%	\tilde{P}&=&\frac{1}{2}\dot{z}^{T}Q^{T}MQ\dot{z} = \frac{1}{2}\dot{x}^{T}\tau_{x}\dot{x}+\frac{1}{2}\dot{\lambda}^{T}\tau_{\lambda}\dot{\lambda}
	%	\end{eqnarray}
	%	The time derivative of the storage function \eqref{eq_const_P} along the system of equations \eqref{maindyn} can be computed as
	%	\beqn
	%	\dot{\tilde{P}}&=&-\dot{x}^\top\nabla_x^2f(x) \dot{x}-\dot{x}^\top  \dot{u}
	%	\leq  -\dot{x}^\top  \dot{u}= \dot u ^{\top}\dot y
	%	\eeqn
	%	which implies that the system \eqref{maindyn} is passive. % with output $-\dot{x}$, input $\dot{u}$. 	Above we assumed $h(x)$ to be  convex and  $f(x)$ as strictly convex. 
	% Further for $u=0$ we have $\dot{\tilde{P}}=0$ $\implies$ $\dot{x}=0$ ( $x$ is some constant). Using this in the first equation of \eqref{maindyn} we get that $\lambda$ is a constant, proving asymptotic stability of $\bar{z}$.% is asymptotically stable.
	%\end{proof}
	%%
	%
	%
	%\subsection{Inequality constraints}\label{ICSS}
	\subsection{Inequality constraints}\label{ICSS}
	%{\em Inequality constraints:} 
	We now define the inequality constraint $g_i(\tilde{u})\leq 0$  as the following hybrid dynamics
	\begin{equation}\label{IED}
	\tau_{\mu}\dot \mu_i=(g_i(\tilde u))^+_{\mu_i}
	\end{equation}
	where  $\tilde{u}\in \mathbb{R}^n$ and $i\in \{1\cdots p\}$. The positive projection of $g_i(\tilde{u})$ can be written as 
	\beqn
	(g_i(\tilde u))_{\mu_i}^+&=&\left\{\begin{matrix}
		g_i(\tilde u) &\mu_i>0\\\max\{0,g_i(\tilde{u})\} & \mu_i=0
	\end{matrix} \right. 
	\eeqn
	Note that the discontinuity in the above equations occurs when $g_i(\tilde{u})<0$ and $\mu_i=0$, the value of $g_i(\tilde{u})^+$ switches from $g_i(\tilde{u})$ to $0$. To make this more visible, we redefine these equations equivalently as follows; 
		\beq 
	(g_i(\tilde u))_{\mu_i}^+&=&\left\{\begin{matrix}
		g_i(\tilde u) &\; (\mu_i>0 \;\;\text{or}\;\; g_i(\tilde{u})>0)\\0 &\;\text{otherwise}
	\end{matrix} \right. 
	\eeq
	%
	%if (i) $(\mu_i>0$ or $g_i(\tilde{u})>0)$ then $(g_i(\tilde u))_{\mu_i}^+=g_i(\tilde u)$  else (ii) $(g_i(\tilde u))_{\mu_i}^+=0$. 
	The projection is said to be active in the second case. Let $\mathcal{P}$ represent the power set of $\{1\cdots p\}$, then we define the function $\sigma:[0,\;\infty)\rightarrow \mathcal{P}$ as follows
	\beq\label{sigma_map} \sigma(t)=\{i \mid \text{ if}\; \mu_i(t)=0\;\;\text{and }\;g_i(\tilde{u})\leq0\,\,\, \forall i \in \{1, ..., p\}\}\eeq
	where the projection is active. With $\sigma(t)$ representing the switching signal, equation \eqref{IED} now takes the form of a switched system 
	\beq \label{active_const_def}
	\tau_{\mu}\dot \mu_i=g_i(\tilde u,\sigma)&=&\left\{\begin{matrix}
		g_i(\tilde u) ;&\; i\notin \sigma(t)\\0 ;&\;i\in \sigma(t)
	\end{matrix} \right. 
	\eeq
	The overall dynamics of the $p$ inequality constraints $g_i(\tilde{u})\leq 0$ $\forall i\in \{1\cdots p\}$ can be written in a compact form as:
	\begin{equation}\label{IEdyn}
	\tau_{\mu}\dot{\mu}=g(\tilde u,\sigma)
	\end{equation}
	where $\mu_i$ and $g_i(\tilde{u},\sigma)$ are $i^{th}$ components of  $\mu$ and $g(\tilde{u},\sigma)$ respectively.
	A strictly passive system can be proven to be asymptotically stable, with Lyapunov function as the storage function. But in the case of switching systems its misleading. 
	It is well known that a sufficient condition for a switched system to be passive system is that the storage function should be common for all the individual subsystems \cite{zhao2006notion}. In general it is not easy to find such storage functions. Here we use passivity property defined with `multiple storage functions'\cite{HybridPassive}. Consider the following storage function(s) 

\beq\label{storage_fun_ineq_const}
	S_{\sigma_q}(\mu)&=&\dfrac{1}{2}\sum_{i\notin\sigma_q}^{}\dot \mu_i^2\tau_{\mu_i}\;\;\; \forall \sigma_q \in \mathcal{P}
	\eeq
	\begin{proposition} \label{prop:ineq_passivity}
		The switched system \eqref{IEdyn} is  passive with multiple storage functions $S_{\sigma_q}$ (defined one for each switching state $\sigma_q\in\mathcal{P}$ ), input port $u_s=\dot{\tilde{u}}$ and  output port $y_s=\dot{\tilde{y}}$ where $\tilde{y}=\sum_{\forall i} \mu_i\nabla_{\tilde u}g_i(\tilde{u})$. That is, for each $\sigma_p\in \mathcal{P}$ with the property that for every pair of switching times $(t_i,t_j)$, $i<j$ such that $\sigma(t_i)=\sigma(t_j)=\sigma_p\in \mathcal{P}$ and $\sigma(t_k)\neq \sigma_p$ for $t_i<t_k<t_j$, we have
		\begin{eqnarray}\label{Hybrid_passivity}
		S_{\sigma_p}(\mu(t_j))-S_{\sigma_p}(\mu(t_i))\leq \int_{t_i}^{t_j}u_s^\top y_sdt
		\end{eqnarray}
		%and additionally $(\mu_{i},\tilde{u})$ converge asymptotically to the equilibrium set containing $(\bar{\mu},\bar{x})$.
	\end{proposition}
	\begin{proposition}\label{prop::ass_stab_ineq}
		The equilibrium set $\Omega_e$ defined by constant control input $\tilde{u}=\tilde{u}^\ast$ of \eqref{IED}
		\beqn\label{IE_equilibrium}
		\Omega_e=\left\{(\bar{\mu},\tilde{u}^\ast)\left|g_{i}(\tilde{u}^{*})\leq 0,\;\; \bar{\mu}_{i}g_{i}(\tilde{u}^{*})=0 \hspace{0.2cm} \forall i \in \{1,\hdots, p\}\right.\right\}
		\eeqn
		is asymptotically stable.
	\end{proposition}
	\subsection{ The overall optimization problem:}\label{sec::overall_opt}
	The most interesting  property of passive systems is their modular nature. One can define power conserving interconnections (such as Newton law's or Kirchoff's current/voltage laws) between these systems, and show that the overall system is passive and there by stability.
	In this subsection we define a power conserving interconnection between passive systems associated with optimization problem with an equality constraint \eqref{maindyn} and an inequality constraint \eqref{IED}.
	\begin{proposition} \label{interconnectpassive}
		% 	The interconnection between \eqref{dyn_cost_fun} and \eqref{IED} defined by 
		Consider the interconnection of passive systems \eqref{maindyn} and  \eqref{IED}, via the following interconnection constraints $ u=-\tilde{y}+v~~ \text{and} ~~\tilde{u}=x, ~ v\in \mathbb{R}^p $. 
		The interconnected system is then passive with port variables $\dot{v}$, $-\dot{x}$. % and .
		Moreover for $v=0$ the interconnected system represents the primal-dual gradient dynamics of the optimization problem
		\begin{equation}\label{SOP_main}
		\begin{aligned}
		& \underset{x \in \mathbb{R}^{n}}{\text{minimize}}
		& & f(x)\\
		& \text{subject to}
		& & h(x)=0 \\
		& & &g_{i}(x)\leq0 \hspace{0.4cm} i = 1,\hdots , p
		\end{aligned}
		\end{equation}
		and the trajectories converge asymptotically to the optimal solution of \eqref{SOP_main}.% with respect to the Lyapunov function $\tilde{S}_{\sigma}(x,\lambda,\mu)$.
	\end{proposition}
	%\begin{proof}
	%Define the storage function $\tilde{S}_{\sigma}(x,\lambda,\mu)=\tilde{P}(x,\lambda)+S_{\sigma}(\mu)$.
	%The time differential of $\tilde{S}_{\sigma}(x,\lambda,\mu)$ is
	%\begin{eqnarray*}
	%\dot{\tilde{S}}_{\sigma}(x,\lambda,\mu)
	%&=&-\dot{u}^\top \dot{x}+\dot{\tilde{u}}^\top\dot{\tilde{y}}\leq -\dot{v}^\top \dot{x}
	%\end{eqnarray*}
	%The interconnection of  \eqref{maindyn} and \eqref{IED}, with $v=0$, gives
	%\begin{eqnarray}\label{primal-dual-dyn}
	%-\tau_{x}\dot{x}&=&\left(\nabla_{x} f(x)+\sum_{i=1}^{m}\lambda_{i}\nabla_{x} h_{i}(x)+\sum_{i=1}^{p}\mu_{i}\nabla_{x} g_{i}(x)\right)\nonumber\\
	%\tau_{\lambda_{i}}\dot{\lambda_{i}}&=& h_{i}(x)\nonumber\\
	%\tau_{\mu}\dot{\mu_{i}}&=& \begin{cases}
	%g_{i}(x) \;\;\;\;\;\;\;\;\;\;\;\;\;\;\text{if}\; \mu_{i}>0 \;\; \forall i \in \{1,\hdots, p\} \label{PriD}\\
	%\text{max}(0,g_{i}(x))\;\; \text{if} \;\mu_{i}=0
	%\end{cases}
	%\end{eqnarray}
	%%which represet the primal-dual gradient dynamics of \eqref{SOP}.
	%Hence the overall system take the form of primal-dual gradient dynamics representing optimization problem with both equality and in-equality constraint \eqref{SOP_main}. \\
	%When $v=0, ~ \dot{\tilde{S}}_{\sigma}(x,\lambda,\mu) \le 0$, for the interconnected system. Stability can thus be concluded using the relation between passivity and stability \cite{l2gain} and Propositions \ref{prop:ineq_passivity}, \ref{prop::ass_stab_ineq}.
	%
	%\end{proof}

	\section{Building Energy Management Formulation}\label{BEMF}
	This section describes the mathematical formulation of the energy management problem of building HVAC systems. The problem is formulated by taking into account the interaction between the multiple consumers and a single producer in achieving social welfare. 
	%\textcolor{black}{The present formulation can be easily extended to multiple producers}. 
	The rising opportunities for demand side flexibility enables the consumers to manage their load to reduce their costs, in this context, we model the coalition by group of consumers in order to have access to wholesale energy markets. The coalition coordinator or energy provider purchases the electricity from wholesale energy markets and resells to each member of the coalition using a simple price structure. In this paper we consider a time-of-use (TOU) pricing. A schematic representation of the interaction between coalition coordinator and group of consumers is shown in Fig. \ref{coal}. Furthermore, each member of the coalition tries to maximize his own benefit by contributing to overall demand reduction. In order to illustrate the energy management problem, we consider a simulated medium sized commercial building with different zones. The zone thermal dynamics is one of the essential component of the modeling building energy systems. % and is given below:
	\begin{figure}
		\center
		\includegraphics[scale=0.45]{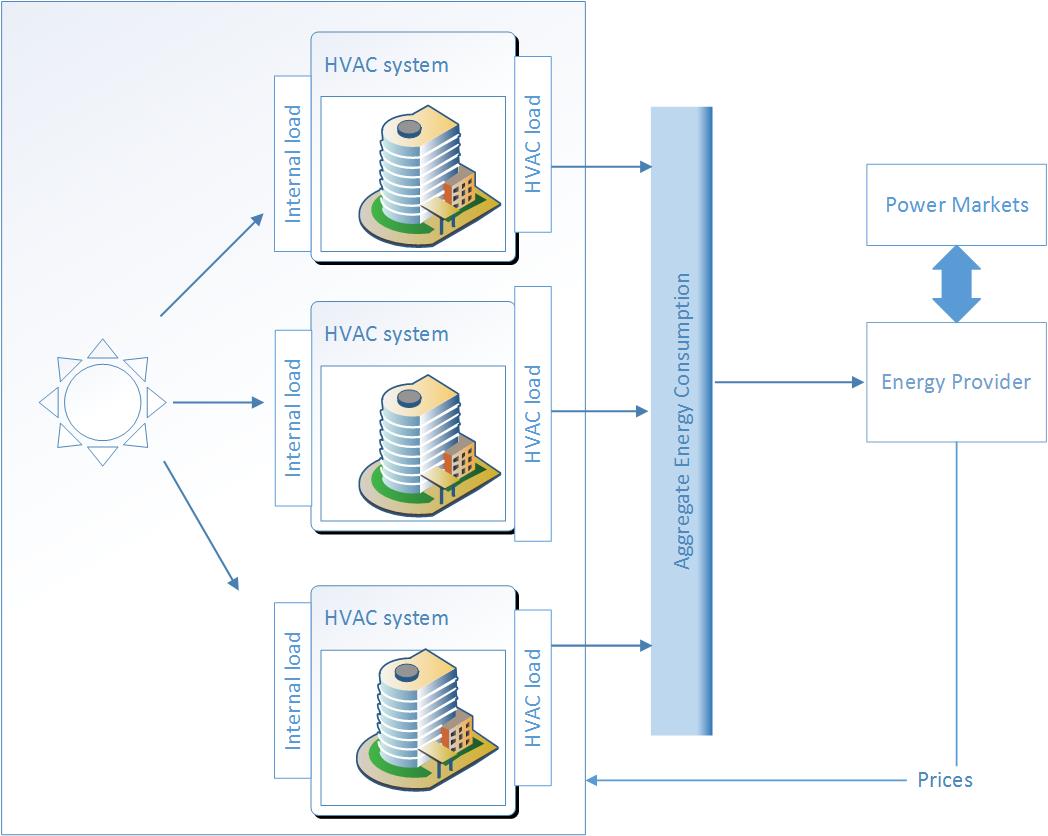}%coalition.jpg}
		\caption{Coalition model of producer-consumer interaction}
		\label{coal}
	\end{figure}
	\subsection{Thermal dynamics of building}
	The thermal dynamics of a multi-zone building can be represented using a resistance-capacitance network and the governing dynamics are given by \cite{chinde2016building}:
	\begin{equation}\label{actualdyn}
	C_{i}\dot{T_{i}}=\sum_{j \in \mathbb{N}_{i}} \frac{(T_{j}-T_{i})}{R_{ij}}%\dot{m}_{i}c_{p}(T_{s}-T_{i})
	+\frac{(T_{\infty}-T_{i})}{R_{i0}}+
	u_{i}+d_{i}
	%\underbrace{c_{p}\dot{m_{i}}(T_{s}-T_{i})}_{u_{i}}+d_{i}
	\end{equation}
	where $\mathbb{N}_{i}$ denotes all resistors connected to the $i^{\text{th}}$ capacitor (includes zone and surface capacitances),  $T_{i}$ is the temperature of the  $i^{\text{th}}$ zone and, $T_{\infty}$ denotes the ambient temperature. $C_{i}$ is the thermal capacitance of the $i^{\text{th}}$ zone, $R_{ij}$ is the thermal resistance between zone $i$ and zone $j$, $R_{i0}$ is the thermal resistance between zone $i$ and ambient conditions, $u_{i}$ is the heating/cooling  input to the zone $i$ and $d_{i}$ denotes the heat gain due to sources such as solar, occupancy etc.
%	
% 	\begin{equation}\label{dyncomp}
% 	C\dot{T}=-BR^{-1}B^{T}T+D_{0}T_{\infty}+Du
% 	\end{equation}
	
% 	Considering steady state thermal dynamics, the above equation \eqref{dyncomp} can be rewritten as
% 	\begin{equation}
% 	AT+b+Du=0
% 	\end{equation}
% 	where $A =- BR^{-1}B^{T}T$ and $b =D_{0}T_{\infty}$.

	\subsection{Problem formulation}
	{\em The optimization problem:}
	\textcolor{black}{The energy management problem is formulated as a group of consumers that form a coalition and a energy provider or coalition coordinator who purchases electricity from the wholesale markets through contracts}. The objective is to define the total welfare function of the coalition of consumers under the operational and market clearing constraints. 
		\textcolor{black}{The optimization problems for each consumer and producer at each time slot is given as below.
	The optimization problem for each consumer $i$ is given as follows
	\begin{equation*}
	\begin{aligned}
	& \underset{x_{i}}{\text{maximize}}
	& & U_{i}(x_{i})-px_{i} \\
	& \text{subject to}
	& &  x_{i}^{min} \leq x_{i} \leq x_{i}^{max}.
	\end{aligned}
	\end{equation*}
	Each of these consumers has a private utility function $U_{i}(x_{i})$, which represents the utility the consumer derives by consumption of $x_{i}$ units of power. $p$ stands for the market price, $U_{i}(x_{i}): \mathbb{R} \rightarrow \mathbb{R}$ is strictly concave function.
	The objective of energy provider is to maximize his profits and is given by
	\begin{equation*}
	\begin{aligned}
	& \underset{\bar{x}\geq 0}{\text{maximize}}
	& & p\bar{x}-U(\bar{x}) \\
	& \text{subject to}
	& &  \bar{x}=\sum_{i =1}^{N}x_{i}
	\end{aligned}
	\end{equation*}
	where, $\bar{x}$ denotes the total supply available to the consumers, N denotes the number of consumers, $U(\bar x): \mathbb{R} \rightarrow \mathbb{R}$ is strictly convex function .
	Once we have the consumer and producer cost functions,} %the social welfare problem can be formulated as the net %benefits of the consumer and producer, and is given by
	the social welfare problem is formulated as the net benefits of the consumers and producer \cite{wei2014distributed}, and is given by
	%---
	%\begin{equation*}
	%\begin{aligned}
	%& \underset{x_{i},\bar{x}}{\text{maximi%ze}}
	%& & \sum_{i=1}^{N}u_{i}(x_{i})-u(\bar{x%}) \\
	%& \text{subject to}
	%& &  \bar{x}=\sum_{i =1}^{N}x_{i}\\
	%& &  & x_{i}^{min} \leq x_{i} \leq %x_{i}^{max}.
	%\end{aligned}
	%\end{equation*}
	\begin{align*}
	& \underset{x_{i},\bar{x}}{\text{maximize}}
	\sum_{i=1}^{N}U_{i}(x_{i})-U(\bar{x}) \\
	& \text{subject to} \hspace{0.2cm}\bar{x}=\sum_{i =1}^{N}x_{i} \hspace{0.4cm} x_{i}^{min} \leq x_{i} \leq x_{i}^{max}
	\end{align*}
	%\end{aligned}
	%\end{equation*}
	%where, $x_{i}$ denotes the power consumption of consumer $i$, %$\bar{x}$ denotes the total supply available to the consumers, %$N$ denotes the number of consumers, $U_{i}(x_{i}): \mathbb{R} %\rightarrow \mathbb{R}$ is strictly concave function, $U(\bar %x): \mathbb{R} \rightarrow \mathbb{R}$ is strictly convex %function.
	
	%Using the generic formulation discussed above, the energy management of  HVAC systems is formulated by taking the discomfort costs as the consumer utility functions and producer utility function %is assumed 
	%to represent a quadratic function\cite{ma2016energy}. The formulation is represented as:
	Using the generic formulation discussed above, the energy management of HVAC system is formulated by considering the discomfort and generation costs as consumer and producer utility functions, respectively \cite{ma2016energy}. This can be formulated as
	\begin{equation}\label{BOPT}
	\begin{aligned}
	& \underset{T_{i},q}{\text{minimize}}
	& & -\left( \sum_{i=1}^{N}U_{i}(T_{i})-U(q)\right)\\
	& \text{subject to}
	& &\sum_{i=1}^{N} \theta \left(\sum_{j \in \mathbb{N}_{i}} \frac{(T_{j}-T_{i})}{R_{ij}} +\frac{(T_{\infty}-T_{i})}{R_{i0}}+d_{i}\right) =q\\
	& & & T_{i}^{min} \leq T_{i} \leq T_{i}^{max} \hspace{0.5cm} i \in \{1, \hdots,N\}\\
	\end{aligned}
	\end{equation}
	where, 
	$U_{i}(T_{i})=b_{i}-\gamma_{i}(T_{i}-T_{i}^{ref})^{2}$, 
	$U(q)=\rho_{1}q^{2}+\rho_{2}q+\rho_{3}$ ($\rho_1 >0$) and $\theta$ denotes the conversion factor from energy consumption to energy demand \cite{henze2003predictive}. The coefficients $\gamma_{i}>0$ determines the tradeoff between cost and comfort \cite{hamalainen2000cooperative}.  The steady state dynamics of \eqref{actualdyn} is considered to relate energy supply and demand. In the compact  notation, the Lagrangian is given as
	%%%\vspace{-0.1cm}
	\begin{eqnarray}\label{Lag}
	\mathcal{L} &=&  U(q)-U(T) + 
	\lambda^{T}(AT+b- q)\nonumber\\
	&&+\mu_{l}^{T}(T^{min}-T)^++\mu_{h}^{T}(T-T^{max})^+
	\end{eqnarray}
	where, 
% 	\begin{eqnarray}
% 	U(T)=\sum_{i=1}^{N}U_{i}(T_{i})
% 	=\sum_{i=1}^{N}b_{i}-\gamma_{i}(T_{i}-T_{i}^{ref})^{2}
% 	\end{eqnarray}
	$U(T)=\sum_{i=1}^{N}U_{i}(T_{i})$. %\nonumber = $\sum_{i=1}^{N}b_{i}-\gamma_{i}(T_{i}-T_{i}^{ref})^{2}$.
	%\hspace$= T^{T}PT+S^{T}T+R$
	%where, $P=diag[(-\gamma_{1}, \hdots,-\gamma_{N} )]$, $S=\begin{bmatrix}
	%2\gamma_{1}T_{1}^{ref}, &\hdots &,2\gamma_{N}T_{1}^{ref}
	%\end{bmatrix}^{T}$, and 
	%$R=\sum_{i=1}^{N}b_{i}-\gamma_{i}(T_{i}^{ref})^{2}$
	As discussed in Section \ref{sec::overall_opt}, the primal dual dynamics of \eqref{Lag} is given as
	\begin{eqnarray}
	\tau_{T}\dot{T}&=&\nabla U(T)-A^{T}\lambda  + \mu_{l}-\mu_{h}\nonumber\\
	\tau_{q}\dot{q}&=&-\nabla U(q)+\lambda \nonumber\\
	\tau_{\lambda}\dot{\lambda} &=&AT+b-q\label{3}\\
	\tau_{\mu_{l}}\dot{\mu_{l}}&=& (T^{min}-T)^{+}_{\mu_{l}}\nonumber\\
	\tau_{\mu_{h}}\dot{\mu_{h}}&=& (T-T^{max})^{+}_{\mu_{h}}\nonumber
	\end{eqnarray}
	%It is observed that the above system from \eqref{1} to %\eqref{5} can be rewritten in the Brayton Moser form as
	%\begin{equation}\label{BMR}
	%Q\dot{x}=\nabla P(x)
	%\end{equation}
	%where  
	%\begin{eqnarray*}
	%Q &=&diag([-\tau_{T},-\tau_{q},\tau_{\lambda},\tau_{\mu%_{l}},\tau_{\mu_{h}}])\\ 
	%P(x)&=&\mathcal{L} 
	%\end{eqnarray*}
	\begin{proposition}
		The primal-dual dynamics \eqref{3} converges asymptotically to the optimal solution of \eqref{BOPT}.
	\end{proposition}
	\begin{proof}
		Since the optimization problem \eqref{BOPT} has a strictly convex cost function and affine constraints, the result follows from Propositions \ref{prop::eq_const} - \ref{interconnectpassive}.
		 %we can conclude that the primal-dual dynamics \eqref{} converges asymptotically to the optimal solution.
		%The dynamics \eqref{1}-\eqref{5} can be represented as an interconnection of passive systems in which \eqref{1}-\eqref{3} form one system with port variables $u_{p} = \nabla g^{T}\mu$ and $y_{p}=-\dot{T}$ which corresponds to dynamics of equality constraint optimization problem as shown in Section \ref{ECSS} and \eqref{4}-\eqref{5} forms the other system with port variables $u_{s}= \dot{T}$ and $y_{s}= u_{p}$ which corresponds to inequality constraint system as shown in \ref{ICSS}. Using the details presented in Section \ref{ICS} the convergence to the optimal solution is guaranteed.
	\end{proof}
	\section{Simulation results}\label{sim}
	In this section, a simulated study is conducted using a building model emulating the ERS test-bed \cite{iec}, which represents a small-sized commercial building as shown in Fig. \ref{ERS}. 
		\begin{figure}[h!]
		\center
		\includegraphics[width=0.7\textwidth]{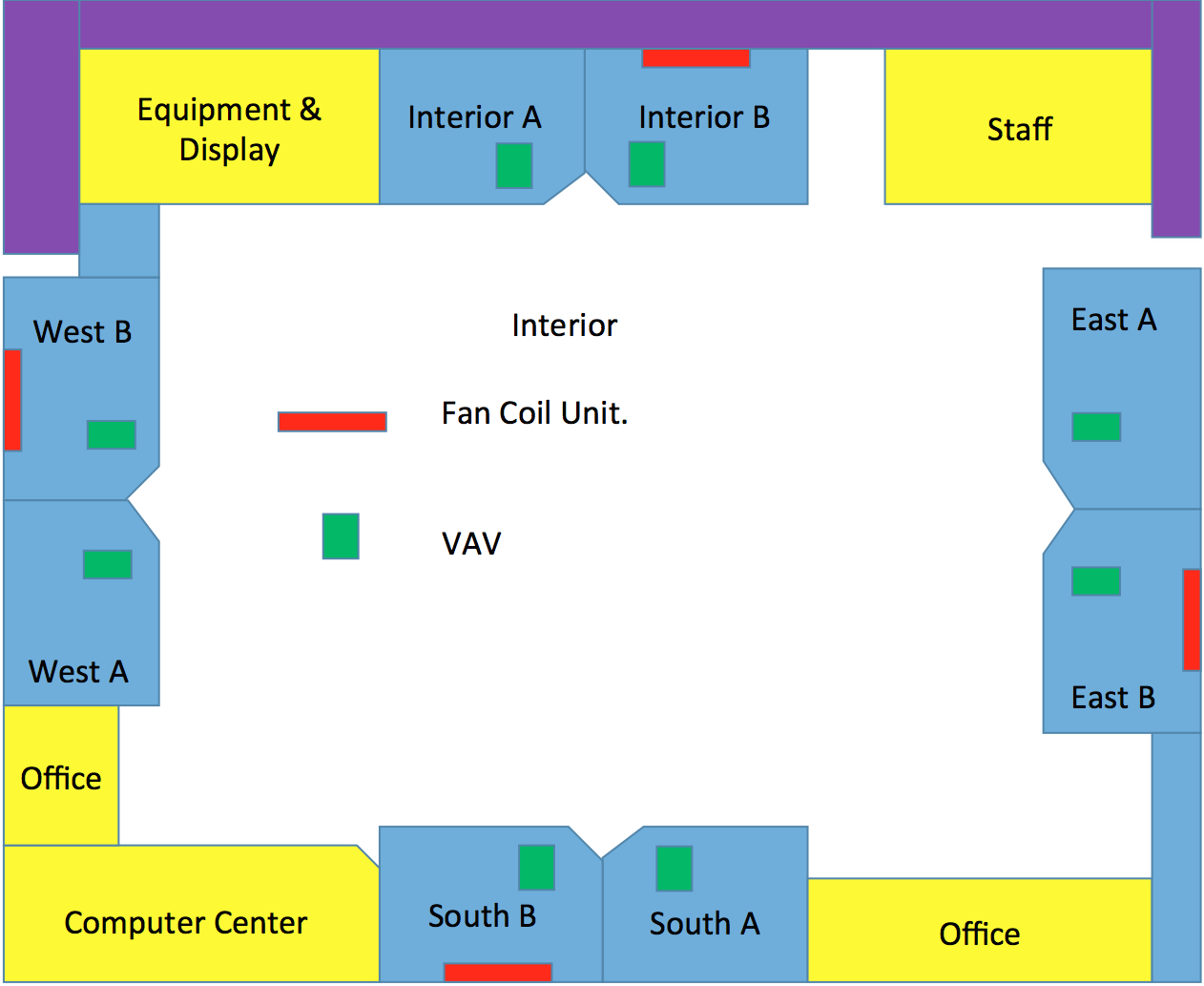}
		\caption{Schematic of the simulated building model}
		%%%\vspace{-1em}
		\label{ERS}
	\end{figure}
		This simulated model consists of two side-by-side
	independent and similar zones marked as A and B and
	distributed in four directions, East, South, West, and North,
	respectively. These zones are served using two air handling
	units (AHU) marked A and B, where each AHU will be serving
	four Variable Air Volume systems (AHU A serving 4
	zones (A) in different directions).
	%This model configuration
	%reflects a medium-sized commercial building and each test
	%zone has approximately 266 sq.ft of floor space. The model
	%consists of three exterior zones and one interior zone (North).
	%In each direction, two constructed zones have identical exposures
	%yielding identical external thermal loads and may have
	%identical internal thermal loads thereby allowing simultaneous,
	%side-by-side comparison testing of many types of HVAC
	%systems and control schemes. Some zones are connected
	%with other offices and spaces.
	For simulation purpose, we consider four zones marked as A,  distributed in four different directions supplied by a single air handling unit (AHU (A)). \textcolor{black}{The parameters used for the simulation is shown in Table \ref{table:parame}.}
	\begin{table}[ht] 
	\caption{PARAMETER SETTINGS} % title of Table
\centering % used for centering table
\begin{tabular}{c c} % centered columns (4 columns)
\hline\hline %inserts double horizontal lines
%Case & Method\#1 \\ [0.5ex] % inserts table
%heading
%\hline % inserts single horizontal line
$\displaystyle T_{\infty}$, $T_{min}$, $T_{max}$, $T_{i}^{ref}$ & 30, 18, 24, 20.5 \\ % inserting body of the table
Inertial time constants ($\tau_{T}$, $\tau_{q}$, $\tau_{\lambda}$,$\tau_{\mu_{l}}$,$\tau_{\mu_{h}}$) & 1 \\
$\rho_{1}$, $\rho_{2}$, $\rho_{3}$ & 0.5, 0, 0  \\
$b_{i}$, $d_{i}$, $R_{i0}$, $\theta$ & 40, 0.5, 11.5, 3  \\
\hline %inserts single line
\end{tabular}
\label{table:parame} % is used to refer this table in the text
\end{table}
	%For simplicity, we assume that there is no heat transfer between the zones.% which eliminates the first term in the equality constraint.
	%is considered for the analysis 
	To illustrate the effect of load reduction during a price surge, we consider a simulated TOU pricing. TOU pricing essentially provides consumers with different rates at different times in a 24 hour period. Fig. \ref{ZT} shows the convergence of the algorithm to its optimal value and also the interplay between supply and demand.

	\begin{remark}
	\textcolor{black}	{	
In this study we have four zones, each zone temperature has an upper and lower bound, giving rise to eight inequality constraints. In Figure \ref{CLS_plot} we have plotted the time evolution of their corresponding Lagrange variables $\mu_i$, $i\in \{1\cdots 8\}$. Let $ \{t_1\leq\cdots\leq t_8$\} denote the ordered sequence of time instances where the Lagrange variables $\mu_i$'s converge to zero.}
		\begin{figure}[h!]
		\center%\textwidth
		\includegraphics[width=0.8\textwidth]{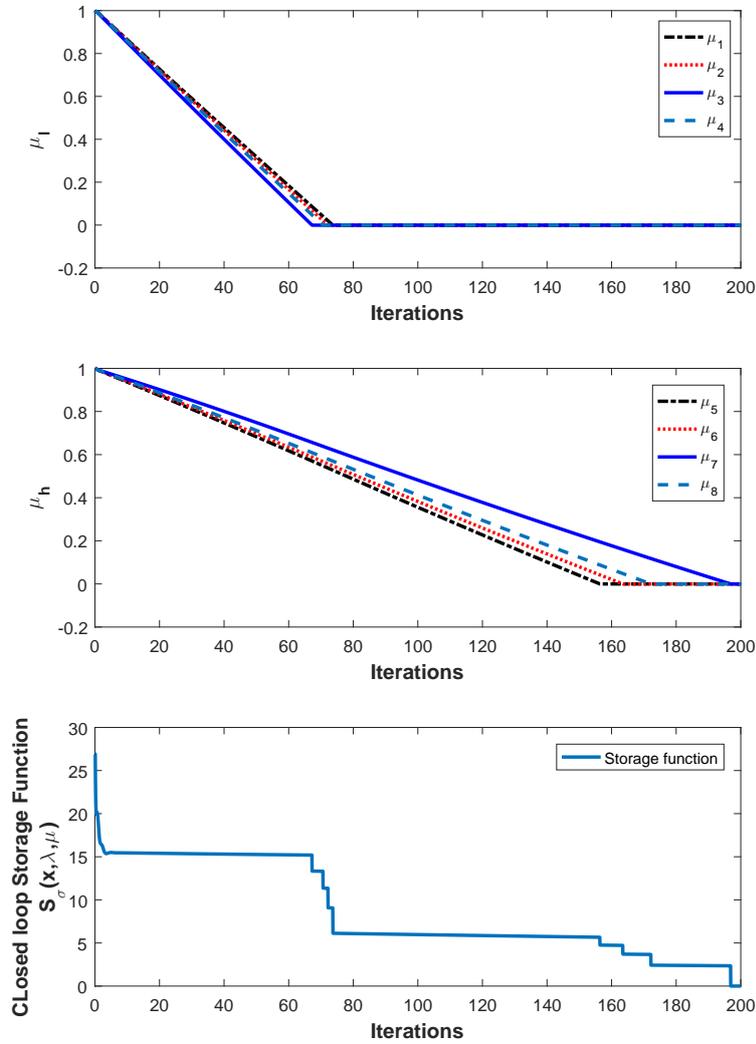}
				\caption{ 	The case study in Section \ref{sim} have four zones, each zone temperature has an upper and lower bound, giving rise to eight inequality constraints. When an inequality constraint is feasible  (i.e. $g_i(T)\leq 0$ see Fig. \ref{ZT}) and its corresponding Lagrange variable ($\mu_i$) converge to zero,  then the closed loop storage function switches to a new storage function that is strictly less than the current one, causing a discontinuity. }
		%\caption{Closed loop storage function $\tilde{S}_{\sigma}(x,\lambda,\mu)=\tilde{P}(x,\lambda)+S_{\sigma}(\mu)$ in Proposition 2.4}
		\vspace{-1em}
		\label{CLS_plot}
	\end{figure}
	\textcolor{black}	{	
The resulting change in the active sets (switching modes) is captured by the switching signal $\sigma(t)$. In the current scenario, the switching signal $\sigma(t)$ have eight different switching modes, and a storage function is defined for each one (refer to Table \ref{table:switching_seq}). Figure \ref{CLS_plot} shows that the closed loop storage function decreases, discontinuously. This discontinuity appears because, at the end of each switching mode, we are switching to a new storage function that is strictly less than the current one. This is coherent with the Proposition \ref{prop:ineq_passivity}, where passivity property is defined with ‘multiple storage functions’.}
	
	\begin{table}[ht] 
\caption{Switching sequence $\sigma(t)$ and the corresponding storage function $S_{\sigma(t)}$} % title of Table
\centering % used for centering table
\begin{tabular}{c c c} % centered columns (4 columns)
\hline\hline %inserts double horizontal lines
%Case & Method\#1 \\ [0.5ex] % inserts table
%heading
%\hline % inserts single horizontal line
$t$ & $\sigma(t)$ & $S_{\sigma(t)}=\dfrac{1}{2}\sum_{i\notin\sigma(t)}^{}\dot \mu_i^2\tau_{\mu}$\\ % inserting body of the table
\hline
$[0, t_1)$   & $\phi$ &$\dfrac{1}{2}\tau_{\mu}\left(\dot{\mu}_1^2+\dot{\mu}_2^2+\dot{\mu}_3^2+\dot{\mu}_4^2+\dot{\mu}_5^2+\dot{\mu}_6^2+\dot{\mu}_7^2+\dot{\mu}_8^2\right)$\\
$[t_1, t_2)$ & $\{3\}$&$\dfrac{1}{2}\tau_{\mu}\left(\dot{\mu}_1^2+\dot{\mu}_2^2+\dot{\mu}_4^2+\dot{\mu}_5^2+\dot{\mu}_6^2+\dot{\mu}_7^2+\dot{\mu}_8^2\right)$\\
%$[t_1, t_2)$ &$\{3,4\}$ \\
$[t_2,t_3)$ & $\{3,4\}$&$\dfrac{1}{2}\tau_{\mu}\left(\dot{\mu}_1^2+\dot{\mu}_2^2+\dot{\mu}_5^2+\dot{\mu}_6^2+\dot{\mu}_7^2+\dot{\mu}_8^2\right)$\\
$[t_3, t_4)$ & $\{3,4,2\}$&$\dfrac{1}{2}\tau_{\mu}\left(\dot{\mu}_1^2+\dot{\mu}_5^2+\dot{\mu}_6^2+\dot{\mu}_7^2+\dot{\mu}_8^2\right)$\\
$[t_4,t_5)$ & $\{3,4,2,1\}$&$\dfrac{1}{2}\tau_{\mu}\left(\dot{\mu}_5^2+\dot{\mu}_6^2+\dot{\mu}_7^2+\dot{\mu}_8^2\right)$\\
$[t_5,t_6)$ & $\{3,4,2,1,5\}$&$\dfrac{1}{2}\tau_{\mu}\left(\dot{\mu}_6^2+\dot{\mu}_7^2+\dot{\mu}_8^2\right)$\\
$[t_6, t_7)$ & $\{3,4,2,1,5,6\}$&$\dfrac{1}{2}\tau_{\mu}\left(\dot{\mu}_7^2+\dot{\mu}_8^2\right)$\\
$[t_7, t_8)$ & $\{3,4,2,1,5,6,8\}$&$\dfrac{1}{2}\tau_{\mu}\dot{\mu}_7^2$\\
$[t_8, \infty)$ &$\{3,4,2,1,5,6,8,7\}$&0\\
\hline %inserts single line
\end{tabular}
\label{table:switching_seq} % is used to refer this table in the text
\end{table} 

\end{remark}

%illustrate that the energy management algorithm converges to its optimal values. 
	%Optimal zone temperature values depend on their comfort weights as well as on the network parameter values. 
	%Fig. \ref{ZT} shows the interplay between supply and demand before attaining the optimal values.

	\begin{figure}[h!] %\hspace{-0.7cm}
	\center
		\includegraphics[width=.85\textwidth]{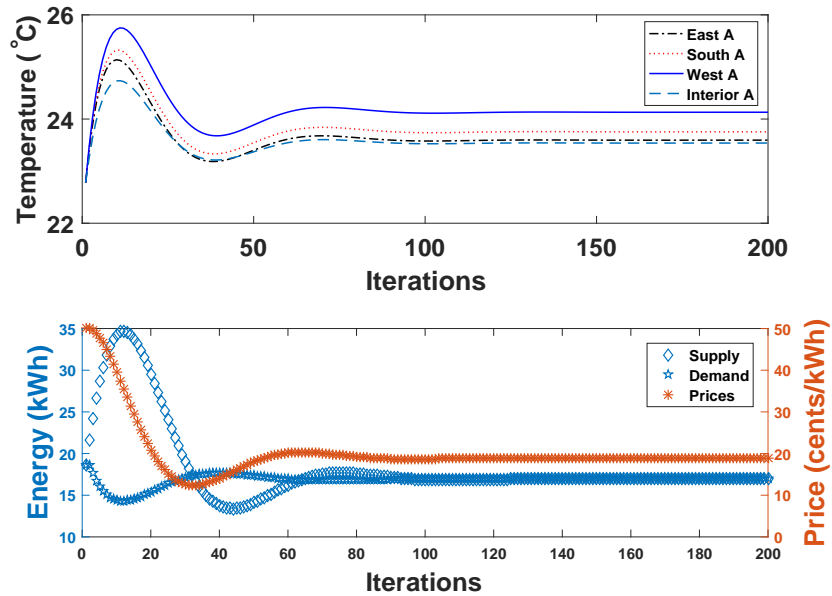}
		%%%\vspace{-1.5em}
		\caption{Zone temperature, Supply ($q$)-demand and pricing profiles }
		%%%\vspace{-1em}
		\label{ZT}
	\end{figure}
	%\begin{figure}[h!]
	%    \includegraphics[width=0.5\textwidth]{SDnew.eps}
	%       \caption{Supply ($q$)-demand}
	%       \label{SD}
	%\end{figure}
	
	In order to evaluate the proposed algorithm for the 24 hour period, we consider the internal load profile as shown in Fig. \ref{IL} which is the sum of heat gains due to occupancy and solar radiation. The occupancy load is computed based on the simulated test bed requirements based on \cite{iec} using the fraction of total occupancy profile. Similarly the solar load is calculated based on the global horizontal irradiance data collected from \cite{wilcox2008users}.
	%For simulation purpose, the solar data during the occupied hours is considered and a time series is generated by considering the hourly data point constant throughout the entire hour.
	The outside air temperature profile \cite{wilcox2008users} for summer is considered as shown in Fig. \ref{TOUZT}. The temperature profile for a particular zone (East A) is shown in Fig. \ref{TOUZT} to illustrate the zone behavior to the TOU pricing for a hot summer day. It can be seen that during the time when prices are high the zone temperatures vary while contributing to the overall demand reduction.
	\begin{figure}[h!]%\hspace{-1cm}
	\center
		\includegraphics[width=12cm,height=6cm]{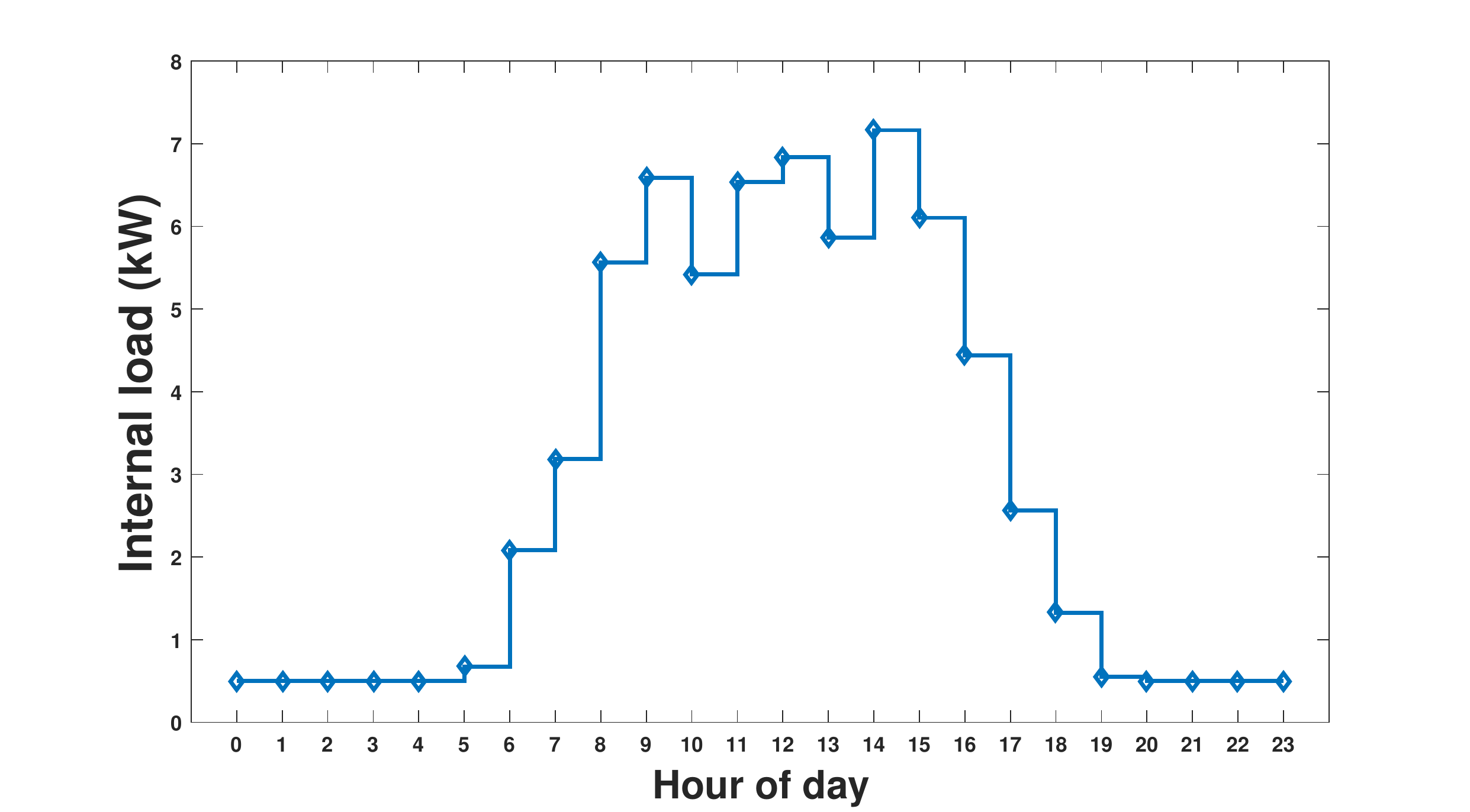}
		%[width=9.5cm,height=3.75cm][width=.7\textwidth]
		\caption{Internal load}
		\label{IL}
		%%%\vspace{-1em}
	\end{figure}
	\begin{figure}[h!]
		\center
		\includegraphics[width=.9\textwidth]{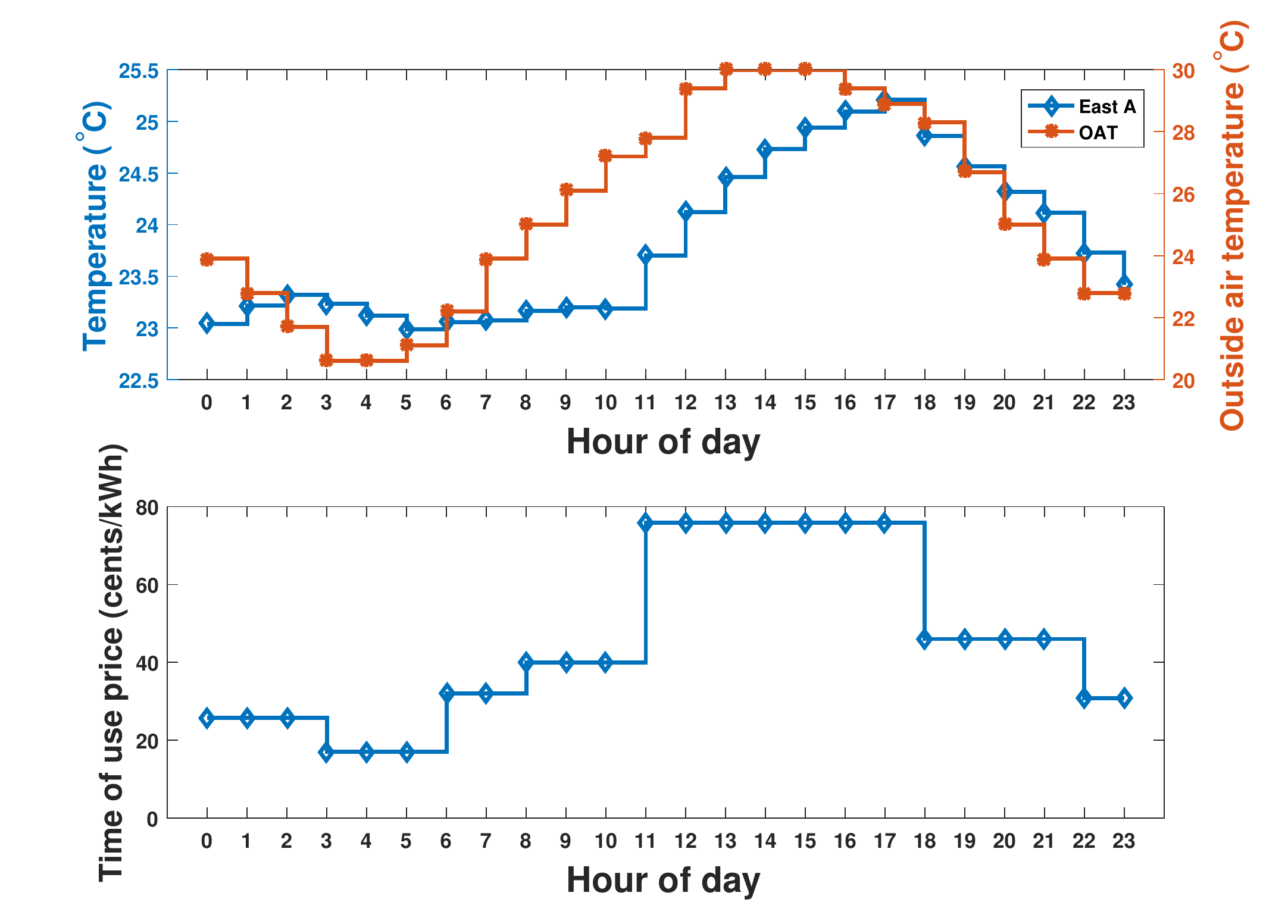}
		%%%\vspace{-1em}
		\caption{Time of use prices and zone temperatures}
		%%%\vspace{-1em}
		\label{TOUZT}
	\end{figure}
		\begin{figure}[h!] 
	\center
		\includegraphics[width=.9\textwidth]{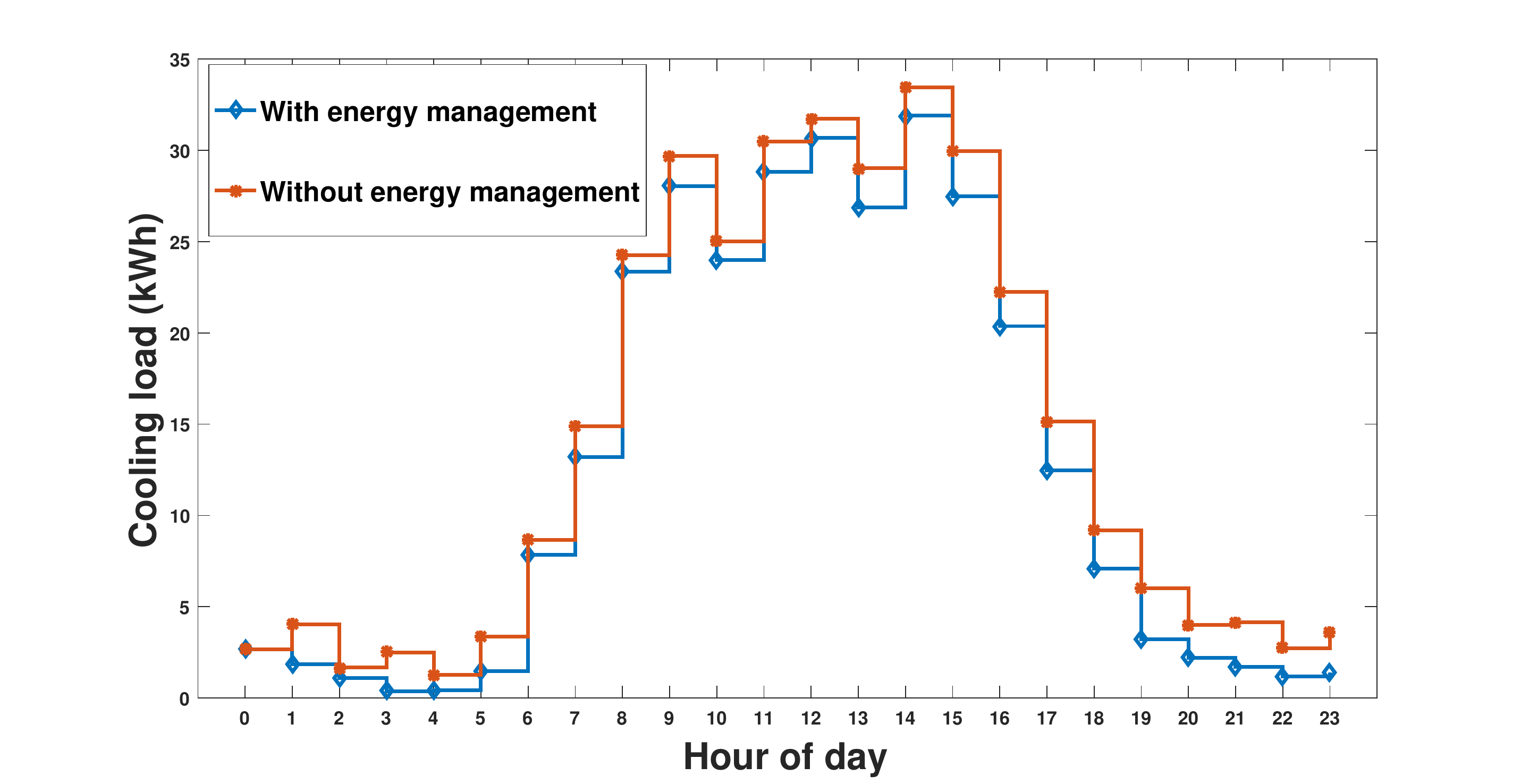}
		%[width=8.7cm,height=5.5cm]
		%%%\vspace{-0.5em}
		\caption{Cooling load}
		\label{CL}
	\end{figure}
	As a result there is a reduction in cooling load of the building as shown in Fig. \ref{CL} in comparison to the building when there is no energy management. Hence, the proposed algorithm effectively reduces the peak load, resulting in overall cost reduction.
	%%%\vspace{-1cm}

	\section{Conclusions}\label{concl}
	%A standard tool for solving convex optimization problems is through primal dual gradient method. 
	Starting from an optimization problem with equality constraint we have shown that their primal-dual equations have a naturally existing Brayton Moser representation. %,  rather than a partial port-Hamiltonian formulation. 
	Using the interconnection properties of BM systems we extended the optimization problem to include inequality constraints. The overall convergence is guaranteed by proving the asymptotic stability of individual subsystems, whose Lyapunov functions derived from BM formulation have their roots in Krasovskii method.
	% In this paper, using the well known passivity tool, we show that the primal dual dynamics can be represented as an interconnection of passive subsystems. The equality constraint subsystem admits a Brayton Moser representation and the convergence guarantees are provided. Later it is extended to include inequality constraints to show the overall interconnected system passive and convergence to the optimal values. 
	This approach is supported by energy management problem in buildings to reduce the overall demand by varying the zone temperature values during the high prices. This approach further lends itself to include the distributed energy resources such as photo-voltaic systems, etc., as well as battery energy storage into the buildings to find the optimal decisions to benefit both consumers and producers.

	\begin{center}
			\section*{APPENDIX} \label{appendix}
	\end{center} 
\subsection*{A.  Proof of Proposition \ref{prop::eq_const}:}
	%\subsection{} %\label{proof::prop::eq_const}
	%\textbf{Proof of Proposition \ref{prop::eq_const}:}
	In BM formulation we represent the system dynamics in pseudo-gradient form, ($Q(z)$ and $P(z)$ are indefinite). Therefore $P(z)$ can not be used as a Lyapunov function for stability analysis. A way of constructing a suitable Lyapunov function involves finding $\alpha\in \mathbb{R} $ and $M\in \mathbb{R}^{n\times n}$ \cite{ortega2003power,ICCvdotidot} such that 
	\begin{eqnarray}\label{common_storage_fun}
	\tilde{P}=\alpha P+\frac{1}{2}\nabla_xP^\top M \nabla_xP.
	\end{eqnarray}
	Considering $\tilde{P}$ \eqref{common_storage_fun} with $\alpha =0$ and $M =\frac{1}{2} diag\{\tau_{x}^{-1},\tau_{\lambda}^{-1}\}$ we have 
	\begin{eqnarray}\label{eq_const_P}
	\tilde{P}&=&\frac{1}{2}\dot{z}^{T}Q^{T}MQ\dot{z} = \frac{1}{2}\dot{x}^{T}\tau_{x}\dot{x}+\frac{1}{2}\dot{\lambda}^{T}\tau_{\lambda}\dot{\lambda}
	\end{eqnarray}
	The time derivative of the storage function \eqref{eq_const_P} along the system of equations \eqref{maindyn} can be computed as
	\beqn
	\dot{\tilde{P}}&=&-\dot{x}^\top\nabla_x^2f(x) \dot{x}-\dot{x}^\top  \dot{u}
	\leq  -\dot{x}^\top  \dot{u}= \dot u ^{\top}\dot y
	\eeqn
	which implies that the system \eqref{maindyn} is passive. % with output $-\dot{x}$, input $\dot{u}$. 	Above we assumed $h(x)$ to be  convex and  $f(x)$ as strictly convex. 
	Further for $u=0$ we have $\dot{\tilde{P}}=0$ $\implies$ $\dot{x}=0$ ( $x$ is some constant). Using this in the first equation of \eqref{maindyn} we get that $\lambda$ is a constant, proving asymptotic stability of $\bar{z}$.\\% is asymptotically stable.

	\subsection*{B. Proof of Proposition \ref{prop:ineq_passivity}} %\label{proof::prop:ineq_passivity}
	%	The vector field for the dynamics \eqref{IEdyn} is not smooth and depends on the discrete state $\sigma$. Hence, 
	We start with analyzing the passivity property for a time interval say $[0\; \tau_{\sigma})$ with fixed $\sigma(t)$. The time derivative of the storage function $S_{\sigma}(\mu)$ is 
	\begin{eqnarray}
	\label{Ssigmadot}
	\begin{aligned}
	\dot{S}_{\sigma}
	&=\sum_{i\notin\sigma}^{}\dot \mu_i\ddot \mu_i\tau_{\mu_i}=\sum_{i\notin\sigma}^{}\dot \mu_i\nabla_{\tilde u}g_i^\top\dot{\tilde u}\\
	&= \dot{\tilde{u}}^\top\left(\dfrac{d}{dt}\sum_{i\notin \sigma}\mu_i\nabla_{\tilde{u}}g_i -\sum_{i\notin \sigma}\mu_i\nabla_{\tilde{u}}^2g_i \dot{\tilde{u}}\right)\nonumber\\
	&= \dot{\tilde{u}}^\top\left(\dot{\tilde{y}} -\sum_{\forall i}\mu_i\nabla_{\tilde{u}}^2g_i \dot{\tilde{u}}\right)\\
	&\leq\dot{\tilde{u}}^\top \dot{\tilde{y}}=u_s^\top y_s.
	\end{aligned}
	\end{eqnarray}
	In step two we use $ \sum_{i\notin\sigma}^{}\mu_i\nabla_{u}g_i=\sum_{\forall i}^{}\mu_i\nabla_{u}g_i$ (which is true since $\mu_i=0$, if $ i\in \sigma $) and in step three we use the convexity of $g$ and non-negativity of the $\mu_i$. The above inequality can be equivalently written as
	\begin{eqnarray}\label{passive_IE}
	S_{\sigma}(\mu(\tau_{\sigma}))-S_{\sigma}(\mu(0))\leq \int_0^{\tau_\sigma}\dot{\tilde{u}}^\top \dot{\tilde{y}} dt
	\end{eqnarray}
	Hence, the system of equation \eqref{IEdyn} represent a finite family of passive systems and \eqref{storage_fun_ineq_const} represents their corresponding storage functions. 
	%Note that in the above inequality, supply rate in the right hand side is independent of $\sigma$ (discrete state), where as the storage functions are dependent on $\sigma$. 
	Since this is not sufficient to prove the passivity property of \eqref{IEdyn}, we further need to analyse the behaviour of the storage functions at all switching times. Let $\sigma(t) \in \mathcal{P}$ denotes current active projection set as defined in \eqref{sigma_map}, then we have the following scenarios: \\
	%This behaviour can be classified into two types, and the overall switching can be a arbitrary combination of these
	\begin{itemize}
		\item[(i)]	 For some $i\notin \sigma(t^-)$, let the projection of $i^{th}$ constraint ($g_i(\tilde{u})\leq 0$) becomes active (i.e $ \mu_i$ reaches $0$ when $g_i(\tilde{u})<0$) at time $t$. This implies a new element $i$ is added to the projection set, $i\in\sigma(t)$. 
		%The storage function \eqref{storage_fun_ineq_const} decreases by loosing the term $\tau_{\mu_i}\dot{\mu}_i^2$ from the summation. 
		The term in the storage function corresponding to this $i$ will not appear in \eqref{storage_fun_ineq_const} as $i\in \sigma(t)$. This happens discontinuously because $g_i(\tilde{u},\sigma)$ switches from $g_i(\tilde{u})< 0$ to $0$. Hence
		\beq\label{inactive2active_constraint} S_{\sigma(t)}(\mu(t))< S_{\sigma(t^-)}(\mu(t^-)) \eeq
		\item[(ii)]  In the case when the projection of an active constraint $i \in \sigma(t^-)$ becomes inactive i.e $i \notin \sigma(t)$, a new term $\tau_{\mu_i}\dot{\mu}_i^2$ is added to the summation of the storage function \eqref{storage_fun_ineq_const}. But this happens in a continuous way because $g_i(\tilde{u},\sigma)$ has to increase from $g_i(\tilde{u})< 0$ to $g_i(\tilde{u})> 0$ by crossing $0$. By continuity argument we have
		\beq\label{active2inactive_constraint} S_{\sigma(t)}(\mu(t))= S_{\sigma(t^-)}(\mu(t^-)) \eeq
	\end{itemize}
	%\end{itemize}
	This situation are depicted in Fig. \ref{CLS_plot} and \ref{SFV}.
	\begin{figure}[h!]
		\center
		\includegraphics[width=0.9\textwidth]{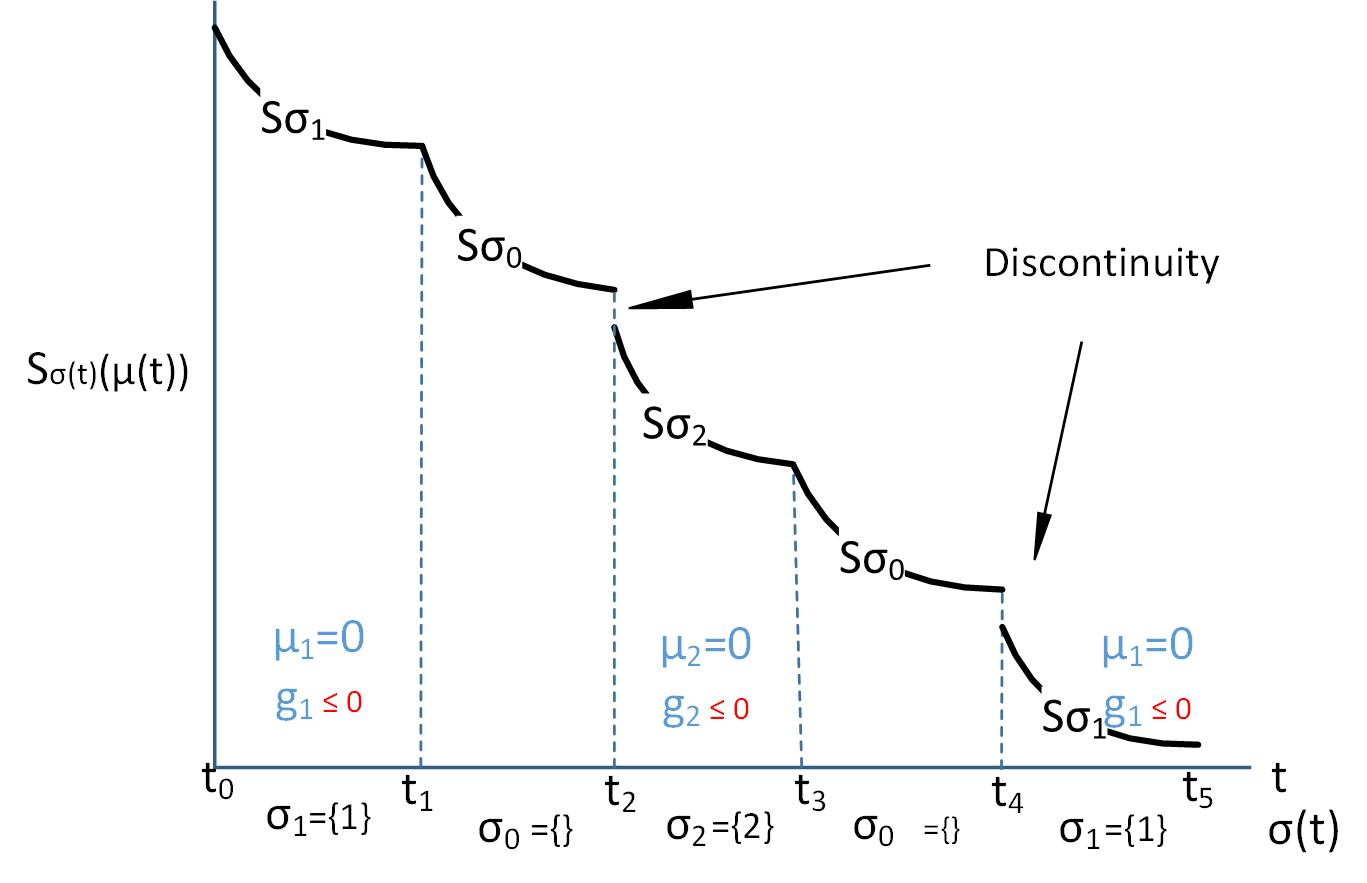}
		\caption{Example for time evolution of storage function with  two inequality constraints ($p=2$). Note that case (i) appears at switching time $t_2$, $t_4$ and case (ii) at $t_1$, $t_3$.}
		\label{SFV}
	\end{figure}
	Now consider a $\sigma_p\in \mathcal{P}$ as given in the proposition.
	Now consider a $\sigma_p\in \mathcal{P}$ with the property that for every pair of switching times $(t_i,t_j)$, $i<j$ such that $\sigma(t_i)=\sigma(t_j)=\sigma_p\in \mathcal{P}$ and $\sigma(t_k)\neq \sigma_p$ for $t_i<t_k<t_j$. 
	We assume that there are $N$ switching times between $t_i$ and $t_j$. %, and denote $\tau =\{t_i,t_{i+1},\cdots, t_{i+N}, t_j\}$. 
	Noting that the storage function is not increasing at switching times we have, 
	%$(j=N+1)$, 
	%  \beqn \tau_1=\{t_i\in \tau|\exists i\in \sigma(t_i)\; \text{and}\; i \notin \sigma(t_i^-) \},\\
	%  \tau_2=\{t_i\in \tau|\exists i\in \sigma(t_i^-)\; \text{and}\; i \notin \sigma(t_i) \}.
	%  \eeqn
	\begin{eqnarray*}
		S_{\sigma(t_j)} &\leq& S_{\sigma(t_j^-)}
		\leq  S_{\sigma(t_{i+N})} +\int_{t_{i+N}}^{t_j}\dot{\tilde{u}}^\top \dot{\tilde{y}} dt\\
		&\leq & S_{\sigma(t_{i+N}^-)} +\int_{t_{i+N}}^{t_j}\dot{\tilde{u}}^\top \dot{\tilde{y}} dt\\
		&\leq & S_{\sigma(t_{i+N-1})} +\int_{t_{i+N-1}}^{t_{i+N}}\dot{\tilde{u}}^\top \dot{\tilde{y}} dt +\int_{t_{i+N}}^{t_j}\dot{\tilde{u}}^\top \dot{\tilde{y}} dt\\
		&\leq & S_{\sigma(t_{i})} +\int_{t_{i}}^{t_{i+1}}\dot{\tilde{u}}^\top \dot{\tilde{y}} dt +\cdots+\int_{t_{i+N}}^{t_j}\dot{\tilde{u}}^\top \dot{\tilde{y}} dt\\
		&= & S_{\sigma(t_{i})} +\int_{t_{i}}^{t_j}\dot{\tilde{u}}^\top \dot{\tilde{y}} dt
	\end{eqnarray*}
	Above we used \eqref{passive_IE}, \eqref{inactive2active_constraint} and \eqref{active2inactive_constraint}. %by noting that the storage function is not increasing at switching times. 
	We thus conclude the system is passive with port variables $(\dot{\tilde{u}}, \dot{\tilde{y}})$. \\
	\subsection*{C. Proof of Proposition \ref{prop::ass_stab_ineq}} %\label{proof::prop::ass_stab_ineq}
	From \eqref{Ssigmadot}, \eqref{inactive2active_constraint} and \eqref{active2inactive_constraint} in Proposition \ref{prop:ineq_passivity}, we can infer that the Lyapunov function \eqref{storage_fun_ineq_const} is non-increasing for a constant $\tilde{u}=\tilde{u}^\ast$, concluding Lyapunov stability.  Now we use hybrid Lasalle's theorem condition \cite{lygeros2003dynamical} to show that $\Omega_e$ is the maximal positively invariant set, defined by 
	\begin{itemize}
		\item[(i)] $\dot{S}_{\sigma}(\mu(t))=0$ for fixed $\sigma$. This is can be verified by substituting $\tilde{u}=\tilde{u}^\ast$ a constant  in \eqref{Ssigmadot}.
		\item[(ii)] 	( $S_{\sigma(t^-)}(\mu(t^-))=S_{\sigma(t)}(\mu(t))$ if $\sigma$ switches between $\sigma(t^-)$ to $\sigma(t)$ at time $t$. In \eqref{IED}, if $g_i(\tilde{u}^\ast)<0$ and the corresponding $\mu_i^\ast>0$ then $\mu_i$ linearly converges to zero, causing a discontinuity in the Lyapunov function $S_{\sigma}(\mu(t))$ ( case-i of Proposition \ref{prop:ineq_passivity}).
		 This does not happen if either
		\beq\label{equild_cond12}
		g_i(\tilde{u}^\ast)<0 ~\text{and} ~ \mu_i^\ast =0~\text{or}~
		g_i(\tilde{u}^\ast)=0 ~\text{and}~  \mu_i^\ast\geq 0
		\eeq
		$g_i(\tilde{u}^\ast)<0$ and $\mu_i^\ast =0$ or $g_i(\tilde{u}^\ast)=0$, $\mu_i^\ast\geq 0$
		because both conditions imply $\dot{\mu}_i=0$. 
	\end{itemize}
	
	%\end{itemize}
	%We now prove that the trajectories of \eqref{IED} are bounded for $\tilde{u}=\tilde{u}^\ast$.
	Consider the quadratic norm $V(\mu)=\frac{1}{2}(\mu-\bar{\mu})^\top \tau_{\mu}(\mu-\bar{\mu})$. Next, using (8), (9) and (17) together with $g^+_i(\tilde{u})_{\mu_i}\leq g_i(\tilde{u})$, we show that the $V(\mu)$ is non-increasing 
	\begin{eqnarray*}
		\dot{V}
		&=&(\mu-\bar{\mu})^\top g^+(\tilde{u}^\ast)_{\mu}
		\leq(\mu-\bar{\mu})^\top g(\tilde{u}^\ast)\\
		&=&\sum_{\forall i\notin \sigma(t)}^{}(\mu_i-\bar{\mu}_i)^\top g_i(\tilde{u}^\ast)+\sum_{\forall i\in \sigma(t)}(\mu_i-\bar{\mu}_i)^\top g_i(\tilde{u}^\ast)\\
		&=&\sum_{\forall i\notin \sigma(t)}(\mu_i-\bar{\mu}_i)^\top g_i(\tilde{u}^\ast)\\
		%&=&\sum_{\forall i\notin \sigma(t)}\mu_i^\top g_i(\tilde{u}^\ast)-\sum_{\forall i\notin \sigma(t)}\bar{\mu}_i^\top g_i(\tilde{u}^\ast)\\
		&=&\sum_{\forall i\notin \sigma(t)}\mu_i^\top g_i(\tilde{u}^\ast)\leq 0
	\end{eqnarray*}
	%Where in step one we used \eqref{IED}, in step two we used the fact that $g^+_i(\tilde{u})\leq g_i(\tilde{u})$, in step three and four we used \eqref{active_const_def} , in step six we used \eqref{equild_cond12} and finally in step seven we again used \eqref{active_const_def}. 
	This implies that the trajectories of \eqref{IED} are bounded for $\tilde{u}=\tilde{u}^\ast$.
	If $g_i(\tilde{u}^\ast)>0$, %from \eqref{IED} then 
	$\mu_i$ increases linearly, contradicting the boundedness of the trajectories. The proof follows by noting that conditions in \eqref{equild_cond12} represent $\Omega_e$ set.
	\subsection*{D. Proof of Proposition \ref{interconnectpassive}} Define the storage function $\tilde{S}_{\sigma}(x,\lambda,\mu)=\tilde{P}(x,\lambda)+S_{\sigma}(\mu)$.
	The time differential of $\tilde{S}_{\sigma}(x,\lambda,\mu)$ is
	\begin{eqnarray*}
		\dot{\tilde{S}}_{\sigma}(x,\lambda,\mu)
		&=&-\dot{u}^\top \dot{x}+\dot{\tilde{u}}^\top\dot{\tilde{y}}\leq -\dot{v}^\top \dot{x}
	\end{eqnarray*}
	The interconnection of  \eqref{maindyn} and \eqref{IED}, with $v=0$, gives
	\begin{eqnarray}\label{primal-dual-dyn}
	-\tau_{x}\dot{x}&=&\left(\nabla_{x} f(x)+\sum_{i=1}^{m}\lambda_{i}\nabla_{x} h_{i}(x)+\sum_{i=1}^{p}\mu_{i}\nabla_{x} g_{i}(x)\right)\nonumber\\
	\tau_{\lambda_{i}}\dot{\lambda_{i}}&=& h_{i}(x)\nonumber\\
	\tau_{\mu}\dot{\mu_{i}}&=& \begin{cases}
	g_{i}(x) \;\;\;\;\;\;\;\;\;\;\;\;\;\;\text{if}\; \mu_{i}>0 \;\; \forall i \in \{1,\hdots, p\} \label{PriD}\\
	\text{max}(0,g_{i}(x))\;\; \text{if} \;\mu_{i}=0
	\end{cases}
	\end{eqnarray}
	which represent the primal-dual gradient dynamics of \eqref{SOP}.
	Hence the overall system take the form of primal-dual gradient dynamics representing optimization problem with both equality and in-equality constraint \eqref{SOP_main}. \\
	When $v=0, ~ \dot{\tilde{S}}_{\sigma}(x,\lambda,\mu) \le 0$, for the interconnected system. Stability can thus be concluded using the relation between passivity and stability \cite{l2gain} and Propositions \ref{prop:ineq_passivity}, \ref{prop::ass_stab_ineq}. 
	\bibliographystyle{IEEEtran}
	\bibliography{refs}
\end{document}